\documentclass[11pt]{amsart}
\usepackage{amsmath}
\usepackage{}
\usepackage[T1]{fontenc}
\usepackage{lmodern}
\usepackage{ifthen}
\usepackage{amsfonts}
\usepackage{amsxtra}
\usepackage{amssymb}
\usepackage{amsthm}
\usepackage{array}
\usepackage[margin=1in]{geometry}
\usepackage{xcolor}
\definecolor{cite}{rgb}{0.50,0.00,1.00}
\definecolor{url}{rgb}{0.00,0.50,0.75}
\definecolor{link}{rgb}{0.00,0.00,0.50}
\usepackage[colorlinks,linkcolor=link,urlcolor=url,citecolor=cite,pagebackref,breaklinks]{hyperref}
\usepackage{mathtools}
\usepackage{mathrsfs}
\usepackage[all]{xy}
\usepackage[lite,abbrev,msc-links]{amsrefs}

\makeindex

\theoremstyle{definition} 
\newtheorem{Unity}{Unity}[section] 
\newtheorem*{defn*}{Definition} 
\newtheorem{defn}[Unity]{Definition} 

\theoremstyle{plain} 
\newtheorem*{thm*}{Theorem}
\newtheorem{thm}[Unity]{Theorem}
\newtheorem{prop}[Unity]{Proposition}
\newtheorem*{cor*}{Corollary}
\newtheorem{cor}[Unity]{Corollary}
\newtheorem{lem}[Unity]{Lemma}
\newtheorem{conj}[Unity]{Conjecture}

\newtheorem{exmp}[Unity]{Example}
\theoremstyle{remark} 
\newtheorem*{rmk*}{Remark}
\newtheorem{rmk}[Unity]{Remark}

\numberwithin{Unity}{section}

\newcommand{\spec}{\textrm{Spec}}

\begin{document}
\title[Strong non-vanishing and strong non-freeness on $n$-Raynaud surfaces]{Strong non-vanishing of cohomologies and strong non-freeness of adjoint line bundles on $n$-Raynaud surfaces}

\author[Yongming Zhang]{Yongming Zhang}
\email{zhangym97@mail.sysu.edu.cn}
\address{School of Science, Sun Yat-sen University(Shenzhen Campus), Shenzhen,
518107, P. R. of China}
\maketitle
\begin{abstract}
We begin by formally defining n-Tango curves and n-Raynaud surfaces. Our investigation then focuses on the pathological behaviors exhibited by n-Raynaud surfaces. As a direct corollary of this analysis, we present a concise disproof of Fujita's conjecture for surfaces in positive characteristics.
\end{abstract}

\section{Introduction}
In \cite{Ra78}, Raynaud constructed a smooth projective surface $X$ over a field of positive characteristic, equipped with an ample line bundle $\mathcal {L}$ satisfying $H^1(X,\mathcal {L}^{-1})\neq0$, thereby providing the first counterexample to the Kodaira vanishing theorem in positive characteristic. As we know, vanishing theorems traditionally play a crucial role in the study of the following celebrated conjecture, originally formulated by Fujita (cf. \cite{Fu85}) in characteristic zero:
\begin{conj}[Fujita's conjecture]\label{Fujita-conj}
Let $X$ be a smooth projective variety of dimension $n$ over an algebraically closed field $\textbf{k}$ and $A$ an ample divisor on $X$. Then:
\begin{enumerate}
\item  for $m\geq n+1$, the adjoint linear system $|K_X + mA|$ is base point free and
\item  for $m\geq n+2$, the adjoint linear system $|K_X + mA|$ is very ample.
\end{enumerate}
\end{conj}

While vanishing theorems are often used in proofs of Fujita's conjecture, they are not strictly necessary. For instance, the conjecture holds for quasi-elliptic surfaces in positive characteristic, even though Kodaira vanishing fails in this setting. This observation suggested that Fujita's conjecture might still hold in positive characteristic despite the failure of Kodaira vanishing, motivating substantial work on alternative techniques for producing global sections of adjoint bundles in this setting. However, contrary to expectations, the author and collaborators constructed explicit counterexamples to Fujita's conjecture in positive characteristic in \cite{GZZ}, revealing that the absence of vanishing theorems can fundamentally undermine the expected behavior of adjoint linear systems.

In this paper, we first formally introduce the definitions of $n$-Tango curves and $n$-Raynaud surfaces. And then we investigate certain pathological properties of an $n$-Raynaud surface, which are determined by its associated vector bundle $\mathcal {E}$ on the base $n$-Tango curve.

Firstly, we find that the strong non-freeness of adjoint line bundles is more deeply connected to the structure of $\mathcal{E}$ (Theorem \ref{main1}) apart from the number $n$. Specifically, there should exist a parameter space of quotients "$\mathcal {E}\twoheadrightarrow \mathcal {L}_0$" of dimension at least $1$, where $\mathcal {L}_0$  is a line bundle of sufficiently small degree. While in \cite{GZZ} the strong non-freeness of the adjoint line bundle arises from the non-surjectivity of certain connecting morphisms in exact sequences.
\begin{thm}(Theorem \ref{main})\label{main1}
Let $C$ be an $n$-Tango curve with an associated vector bundle $\mathcal{E}$ of rank $2$, and let $\psi:X\stackrel{l:1}\rightarrow\mathbb{P}(\mathcal {E})$ be the corresponding $n$-Raynaud surface constructed in section \ref{preliminary}. For any integer $m=lq+r \in \mathbb{N}^+$ where $0\leq r<l$,
suppose that there exists a surjective morphism $\sigma_0:\mathcal{E}\twoheadrightarrow \mathcal {L}_0$ to a line bundle $\mathcal {L}_0$  on $C$ satisfying
\begin{enumerate}
  \item [(1)] $\dim H^0(C,\mathcal{E}^{\vee}\otimes \mathcal {L}_0)\geq2$ and
  \item [(2)] $ H^0(C,\omega_C\otimes {\mathcal {L}_0^{-q}(-Q)})\neq 0$ for some divisor $Q$ of positive degree on $C$.
\end{enumerate}
Then $q< p^n$ and there exists a nonempty open subset $C_0\subset C$ such that for every point $P\in C_0$ the ample line bundle $\mathcal{O}_X(m\tilde{S} + \phi^*(Q+P))$ has base point $\phi^{-1}(P)\cap \widetilde{T}$ on $X$.

In particular, if the condition (2) is replaced by
 \begin{enumerate}
   \item [(2*)] $H^0(C,\omega_C\otimes {\mathcal {L}_0^{-(p^n-1-d)}(-(p^n+l)N-Q)})\neq 0$ for some divisor $Q$ of positive degree,
 \end{enumerate}
 then  the same conclusion holds for the adjoint line bundle $\mathcal{O}_X(K_X +r\tilde{S} + \phi^*(Q+P))$.
\end{thm}
Secondly, as a corollary we provide a concise disproof of Fujita's conjecture (Corollary \ref{fujita1}), bypassing the extensive computations of those connecting map required in \cite{GZZ}.
\begin{cor}(\cite[Theorem 1.2] {GZZ}\label{fujita1})
For any integer $r>0$, there exists a smooth projective surface $X$ with an ample divisor $A$   such that the adjoint linear system
$$|K_X+r A |$$ has base points.
\end{cor}
Finally, we establish a strong form of Kodaira non-vanishing in Theorem \ref{strongnonvanishing1}, whose behavior only depends on the number $n$ or the degree of the associated line bundle $\mathcal{L}$ (or $\mathcal{N}$) on the base curve.
\begin{thm}[Theorem \ref{strongnonvanishing}]\label{strongnonvanishing1}
For any integer $m>0$, there exists a smooth projective surface $X$ and an ample line bundle $\mathcal{H}$ on $X$ such that $H^1(X,\mathcal{H}^{-p^m})\neq0$
\end{thm}


\textbf{Acknowledgement:} The author is deeply grateful to his advisor, Professor Xiaotao Sun, and postdoctoral advisor, Professor Meng Chen, for their continuous encouragement and guidance in exploring mathematical problems. Special thanks are also due to Yifei Chen and Jie Shu for their useful discussions to this work.
\section{$n$-Tango curve and $n$-Raynaud surface}\label{preliminary}
Throughout this paper, we fix an algebraically closed field $\mathbf{k}$ of characteristic $p>0$
\subsection{$n$-Tango curve}\label{n-Tango}
Let $C$ be a smooth projective curve defined over $\mathbf{k}$ with the function field $K(C)$. We denote by
$K(C)^{p}=\left\{f^{p} | f \in K(C)\right\}$ the subfield of $p$-th powers. Let $F$ denote the absolute Frobenius morphism. In \cite{Ra78,Tango72} the following exact sequence
$$0 \rightarrow  \mathcal{O}_{C}  \rightarrow F_{*}\mathcal{O}_{C}\rightarrow \mathcal{B}^{1}\rightarrow 0$$ is used to construct Tango curves,
where $\mathcal{B}^{1}$ is the sheaf of exact $1$-forms on $C$.
More generally, we consider the following exact sequence
$$0 \rightarrow  \mathcal{O}_{C}  \rightarrow F_{*}^{n}\mathcal{O}_{C}\rightarrow F_{*}^{n}\mathcal{O}_{C}/\mathcal{O}_{C} \rightarrow 0$$
to give the definition of $n$-Tango curve.
\begin{defn}
A smooth projective curve $C$ over $\mathbf{k}$ is called an \textbf{$n$-Tango curve} if it satisfies the following conditions.
\begin{enumerate}
  \item There exists a rational function $f \in K(C) \backslash K(C)^{p}$ such that $(\mathrm{d} f)=p^n D$ for some divisor $D$ on $C$ with $\deg D>0$ and some integer $n>0$. Denote by the associated line bundle $\mathcal {L}=\mathcal{O}_{C}(D)$, then $\omega_C\simeq \mathcal {L}^{p^n}$ and we have a nonzero section $s_0\in H^0(C, F_*^{n-1}\mathcal{B}^{1}(-D))$.
  \item Moveover, we assume that this section lifts to a section $s\in H^0(C,(F_*^{n}\mathcal {O}_C/\mathcal {O}_C)(-D))$ via the natural quotient map $F_*^{n}\mathcal {O}_C/\mathcal {O}_C\twoheadrightarrow F_*^{n-1}\mathcal{B}^{1}$.
\end{enumerate}
A triple $(C,f,D)$ satisfying these conditions is called an \textbf{$n$-Tango data}.
\end{defn}
\subsection{Local analysis and vector bundle construction}
Let $(C,f,D)$ be an $n$-Tango data, then by definition we have a nonzero section $s_0\in H^0(C, F_*^{n-1}\mathcal{B}^{1}\otimes \mathcal {L}^{-1})$.
Take an affine open covering $C=U_1\cup U_2$ such that $\mathcal {L}|_{U_i}$ is trivial with the generators $\eta_i \in H^0 (U_i ,\mathcal {L}|_{U_i})$ and the transition relation $\eta_1=\alpha\eta_2$ for some $\alpha \in \Gamma(U_1\cap U_2,\mathcal {O}_C)^*$. Via the nature morphisms $$F_*^{n}\mathcal {O}_C\stackrel{\psi}{\twoheadrightarrow }F_*^{n}\mathcal {O}_C/\mathcal {O}_C\stackrel{\phi}{\twoheadrightarrow} F_*^{n-1}\mathcal{B}^{1},$$ there exist two regular functions $z_i\in \Gamma(\mathcal {O}_C,U_i)$ such that $s_0|_{U_i}=\phi\circ\psi(\sqrt[p^n]{z_i})\otimes \frac{1}{\eta_i}$ and the compatibility condition yields $\phi\circ\psi(\sqrt[p^n]{z_1})=\alpha\phi\circ\psi(\sqrt[p^n]{z_2})$.
By condition (2) $s_0$ lifts to $s\in H^0(C,(F_*^{n}\mathcal {O}_C/\mathcal {O}_C)\otimes \mathcal {L}^{-1})$ via $\phi$, giving $s|_{U_i}=\psi(\sqrt[p^n]{z_i})\otimes \frac{1}{\eta_i}$ and the relation $\psi(\sqrt[p^n]{z_1})=\alpha\psi(\sqrt[p^n]{z_2})$. Hence we have the relation
$$\sqrt[p^n]{z_1}=\alpha\sqrt[p^n]{z_2}+\beta$$ for some $\beta\in\Gamma(U_1\cap U_2,\mathcal {O}_C)$.

Moreover, we get a sub-sheaf $\mathcal{L}\hookrightarrow F_*^{n}\mathcal {O}_C/\mathcal {O}_C$ and then a locally free sub-sheaf of rank two $\mathcal{E}:=\psi^{-1}(\mathcal{L})\subset F_*^{n}\mathcal {O}_C$ from the diagram
$$
\xymatrix{
  0  \ar[r]&\mathcal {O}_C\ar[r] & F_*^{n}\mathcal {O}_C \ar[r]^{\psi} &F_*^{n}\mathcal {O}_C/\mathcal {O}_C   \ar[r]^{ } & 0  \\
  0  \ar[r] & \mathcal {O}_C \ar@{=}[u] \ar[r] &\mathcal{E} \ar@{^{(}->}[u] \ar[r]^{} & \mathcal{L}\ar@{^{(}->}[u] \ar[r] & 0.
   }
$$
Locally, $$\mathcal{E}|_{U_i}=\mathcal {O}_{U_i}\cdot 1\oplus\mathcal {O}_{U_i}\cdot \sqrt[p^n]{z_i},$$ with translation relation  $\sqrt[p^n]{z_1}=\alpha\sqrt[p^n]{z_2}+\beta$ for some $\beta\in\Gamma(U_1\cap U_2,\mathcal {O}_C)$. I.e. the vector bundle $\mathcal{E}$ is defined by the transition matrix $\left(
   \begin{array}{cc}
       1 & \beta \\
     0 & \alpha \\
    \end{array}
  \right)\in GL(2,\mathcal {O}_{U_1\cap U_2})$ .

From the above argument we have the following relation of rational functions $$z_1=\alpha^{p^n}z_2+\beta^{p^n}$$ with $z_i\in \Gamma(\mathcal {O}_C,U_i)$, $\alpha \in \Gamma(U_1\cap U_2,\mathcal {O}_C)^*$ and $\beta\in\Gamma(U_1\cap U_2,\mathcal {O}_C)$.
Taking differentials, $\mathrm{d}z_1=\alpha^{p^n}\mathrm{d}z_2$, implying $\omega_C$ contains a sub-sheaf locally generated by $\mathrm{d}z_i$, isomorphic to $\mathcal {L}^{p^n}$. Since $\omega_C\simeq \mathcal {L}^{p^n}$, $\mathrm{d}z_i$ generates $\omega_C$ locally, making $z_i$ a local parameter on $U_i$. Consequently, the nature map $\emph{Sym}^{m}(\mathcal{E})\rightarrow F_*^{n}\mathcal {O}_C$ is an embedding for $m<p^n$ and an isomorphism for $m=p^n$.

In summary, we establish the following proposition.
\begin{prop}\label{TC}
The following conditions are equivalent for an $n$-Tango data $(C,f,D)$:
\begin{enumerate}
  \item there exists an open affine cover $C=U_1\cup U_2$ and relation of rational functions $$z_1=\alpha^{p^n}z_2+\beta^{p^n}$$ where $z_i\in \mathcal {O}_{U_i}$, $\alpha \in \Gamma(U_1\cap U_2,\mathcal {O}_C)^*$ and $\beta\in\Gamma(U_1\cap U_2,\mathcal {O}_C)$ and the canonical sheaf $\omega_C$ is locally generated by $\mathrm{d}z_i$ on $U_i$.
  \item There is a line bundle $\mathcal{L}$ satisfies $\omega_C\simeq \mathcal {L}^{p^n}$, and there exists a rank $2$ vector bundle $\mathcal{E}\subset F_*^{n}\mathcal {O}_C$ fitting into the exact sequence:
      $$0\rightarrow\mathcal{O}_C\rightarrow\mathcal{E}\rightarrow \mathcal{L}\rightarrow0$$
  with transition matrix $\left(
   \begin{array}{cc}
       1 & \beta \\
     0 & \alpha \\
    \end{array}
  \right)\in GL(2,\mathcal {O}_{U_1\cap U_2})$ over $U_1\cap U_2$.
\end{enumerate}

\end{prop}
\subsection{An motivation for defining $n$-Tango curve}\label{ss2.3}
Given an $n$-Tango curve $C$ with an associated section $s_0\in H^0(C, F_*^{n-1}\mathcal{B}^{1}\otimes \mathcal{L}^{-1})$ and its lift $s\in H^0(C,(F_*^{n}\mathcal {O}_C/\mathcal {O}_C)(-D))$. Let's consider the following diagram:
$$
\xymatrix{
  0  \ar[r]^{ } & (F_*^{n-1}\mathcal {O}_C/\mathcal {O}_C)\otimes \mathcal{L}^{-1}  \ar[r]^{ } &(F_*^{n}\mathcal {O}_C/\mathcal {O}_C)\otimes \mathcal{L}^{-1}   \ar[r]^{ } & F_*^{n-1}\mathcal{B}^{1}\otimes \mathcal{L}^{-1}   \ar[r]^{ } & 0  \\
  0  \ar[r] &  F_*^{n-1}\mathcal {O}_C \otimes \mathcal{L}^{-1}  \ar@{->>}[u] \ar[r]^{ } &F_*^{n}\mathcal {O}_C\otimes \mathcal{L}^{-1} \ar@{->>}[u] \ar[r]^{} & F_*^{n-1}\mathcal{B}^{1}\otimes \mathcal{L}^{-1} \ar@{=}[u] \ar[r]^{} & 0 \\
  &\mathcal {O}_C\otimes \mathcal{L}^{-1}  \ar@{^{(}->}[u] \ar@{=}[r] &\mathcal {O}_C\otimes \mathcal{L}^{-1} . \ar@{^{(}->}[u]  &  &
   }
$$
The second exact column of the above diagram induces the long exact sequence
$$0\rightarrow  H^0(C,(F_*^{n-1}\mathcal {O}_C/\mathcal {O}_C)\otimes \mathcal{L}^{-1}) \stackrel{\delta}\rightarrow H^1(C, \mathcal{L}^{-1}) \stackrel{F^{*n}}\rightarrow H^1(C,F_*^{n}\mathcal {O}_C\otimes \mathcal{L}^{-1}) \rightarrow.$$
this yields a nonzero element $\delta(s)\in H^1(C, \mathcal{L}^{-1})$ satisfying $F^{*n}(\delta(s))=0$, which corresponds to the locally free sheaf $\mathcal{E}$ and will be used to construct the ruled surface in the next section; while from the second exact row we derive another exact sequence $$0\rightarrow  H^0(C, F_*^{n-1}\mathcal{B}^{1}\otimes \mathcal{L}^{-1}) \stackrel{\epsilon}\rightarrow H^1(C,F_*^{n-1}\mathcal {O}_C \otimes \mathcal{L}^{-1}) \stackrel{F^{*}}\rightarrow H^1(C,F_*^{n}\mathcal {O}_C\otimes \mathcal{L}^{-1}) \rightarrow\cdots.$$
Since $s\in H^0(C,(F_*^{n}\mathcal {O}_C/\mathcal {O}_C)\otimes \mathcal{L}^{-1})$ is a lift via the natural quotient $F_*^{n}\mathcal {O}_C/\mathcal {O}_C\twoheadrightarrow F_*^{n-1}\mathcal{B}^{1}$, we have $F^{*n-1}(\delta(s))=\epsilon(s)\neq0$. This confirms $n$ is the smallest integer such that $F^{*n}(\delta(s))=0$. Therefore, we obtain a locally free sheaf $\mathcal{E}$ which splits after an $n$-th pullback via Frobenius morphism (see subsection \ref{ruledsurface}).

\begin{rmk}\label{rmk1}
For $n=1$, the condition (2) is satisfied automatically by (1) recovering the standard definition of Tango curve (cf. \cite{Tango72,Ra78,Mu13}). When $n>1$, the condition (2) is essential: there exist such triples $(C,f,D)$ only satisfying the condition (1) but not the condition (2). 
\end{rmk}
As in the base curve construction in \cite[section 2.2]{GZZ}, the following example is a slight modification of the example 1.3 in \cite{Mu13}. This construction originates from Gieseker's work in \cite{G71} for the case $e=n=1$ and $p=3$.
\begin{exmp}\label{Tangocurve}
Let $Q(X,Y)$ be a homogeneous polynomial in two variables of degree $e$ with nonzero coefficient of $Y^e$ and $C\subset \mathbb{P}^2=\textbf{Proj}\ k[X,Y,Z]$ be the curve defined by the homogeneous equation of degree $p^ne$:
$$Q(X^{p^n},Y^{p^n})-X^{p^ne-1}Y=Z^{p^ne-1}X.$$

The curve $C$ is smooth and intersects $X=0$ precisely at the point $\infty=[0:0:1]$ with multiplicity $p^ne$. Define the affine chart $$U_1:=C\setminus \infty =\textbf{Spec}\ k[y_1,z_1]/(Q(1,y_1^{p^n})-y_1-z_1^{p^ne-1})$$ where $y_1=\frac{Y}{X}$ and $z_1=\frac{Z}{X}$. On $U_1$, the relation
$-\mathrm{d}y_1=-(z_1)^{p^ne-2}\mathrm{d}z_1$ holds, so $\omega_C|_{U_1}$ is generated by $\mathrm{d}z_1$. The degree of the canonical divisor $\deg \omega_C=p^ne(p^ne-3)$, hence $(\mathrm{d}z_1)=p^ne(p^ne-3)\infty$. Let $D:=e(p^ne-3)\infty$ and $\mathcal {L}=\mathcal{O}_{C}(D)$. This defines a triple $(C,z_1,D)$, a sub-line bundle $\mathcal {L}^{p^{n-1}}\hookrightarrow\mathcal{B}^{1}$ and hence a nonzero section $s_0\in H^0(C, F_*^{n-1}\mathcal{B}^{1}\otimes \mathcal {L}^{-1})$.

 Let $U_2=C\cap \{Z\neq 0\}\subset C$ be an open affine chart containing $\infty$ defined by the equation $$Q(x^{p^n},y^{p^n})-x^{p^ne-1}y=x$$ where $y=Y/Z$ and $x=X/Z$. Differentiating both sides yields $-x^{p^ne-1}\mathrm{d}y=(1-yx^{p^ne-2})\mathrm{d}x$. Note that the special point $\infty$ is given by  $x=y=0$. Shrinking $U_2$ to a neighborhood of $\infty$ where $1-yx^{p^ne-2}\neq 0$, $\omega_C|_{U_2}$ is generated by $\mathrm{d}y$
and $y$ is a local parameter at $\infty$. The ideal $(x)=(y^{p^ne})\subset \mathcal {O}_{C,\infty}$ gives $$v_\infty(x)=v_\infty(y^{p^ne})=p^ne.$$
Define  $$z_2:=\frac{x^{p^ne-2}}{Q(x^{p^n},y^{p^n})y^{p^ne(p^ne-3)}}\cdot y,$$ whose differential is $$\mathrm{d}z_2=\frac{ x^{p^ne-2}(1+y x^{p^ne-2})}{Q(x^{p^n},y^{p^n})y^{p^ne(p^ne-3)}(1-yx^{p^ne-2})}\mathrm{d}y.$$
 It is easy to check that $$v_\infty(\frac{x^{p^ne-2}}{Q(x^{p^n},y^{p^n})y^{p^ne(p^ne-3)}})=v_\infty(\frac{ x^{p^ne-2}(1+y x^{p^ne-2})}{Q(x^{p^n},y^{p^n})y^{p^ne(p^ne-3)}(1-yx^{p^ne-2})})=0.$$
 So $\mathrm{d}z_2$ generates  $\omega_{C}|_{U_2}$ on sufficiently small $U_2$.
And the key relation
$$z_1-\frac{1}{Q(x^{p^n},y^{p^n})}=\frac{1}{x}-\frac{1}{Q(x^{p^n},y^{p^n})}=\frac{x^{p^ne-1}y}{x Q(x^{p^n},y^{p^n})}=y^{p^ne(p^ne-3)}z_2$$
simplifies to
$$z_1=(y^{e(p^ne-3)})^{p^n}z_2+(Q^{-1}(x,y))^{p^n}.$$
After shrinking $U_2$ to ensure $y^{e(p^ne-3)}\in \mathcal {O}(U_1\cap U_2)^*$ and $Q^{-1}(x,y)\in \mathcal {O}(U_1\cap U_2)$. The cover $C=U_1\cup U_2$ satisfies the rational function relation with $\omega_C$ locally generated by $\mathrm{d}z_1$ and $\mathrm{d}z_2$ on $U_1$ and $U_2$ respectively.
Therefore, triple $(C,z_1,D)$ is an $n$-Tango data by Proposition \ref{TC}.
\end{exmp}

\subsection{Ruled surface over $n$-Tango curve}\label{ruledsurface}

Let $C$ be an $n$-Tango curve with an associated divisor $D$ on $C$ and $\mathcal {L}=\mathcal {O}(D)$. And
let $s_0\in H^0(C, F_*^{n-1}\mathcal{B}^{1}\otimes \mathcal {L}^{-1})$ be the associated section, which lifts to a section $s\in H^0(C,(F_*^{n}\mathcal {O}_C/\mathcal {O}_C)\otimes \mathcal{L}^{-1})$. By subsection \ref{ss2.3}, we get an element $0\neq \delta(s)\in H^1(C, \mathcal {L}^{-1})$ satisfying $F^{*n}(\delta(s))=0$ and $F^{*n-1}(\delta(s))\neq0$.
So $\delta(s)\in H^1(C, \mathcal {L}^{-1})$ gives a non-trivial extension
$$0\rightarrow \mathcal{O}_{C}\rightarrow \mathcal {E}\rightarrow \mathcal {L}\rightarrow 0.\eqno(*),$$
and $n$ is the smallest integer such that
$$\xymatrix@C=0.5cm{
  0 \ar[r] & \mathcal{O}_{C} \ar[r] & F^{n*}\mathcal {E} \ar[r] & F^{n*}\mathcal {L} \ar[r]\ar@/_/[l]_{\tau} & 0 }
\eqno(**)$$
splits under the Frobenius pullback $F^{n*}$.
Setting $\mathcal {E}^{(p^n)}=F^{n*}\mathcal {E}$, consider the following diagram
$$\xymatrix{
  \mathbb{P}(\mathcal {E}) \ar@/_/[ddr]_{\pi} \ar@/^/[drr]^{F^n}
    \ar@{.>}[dr]|-{F_1}                   \\
   & \mathbb{P}(\mathcal {E}^{(p^n)}) \ar[d]^{\pi_1} \ar[r]_{F_2}
                      & \mathbb{P}(\mathcal {E}) \ar[d]_{\pi}    \\
   & C \ar[r]^{F^n}     & C.               }
\eqno(***)$$
Let $S\subseteq \mathbb{P}(\mathcal {E})$ be the section corresponding to the exact sequence ($*$). The splitting $\tau$ of ($**$) induces a section $T^{\prime}\subseteq \mathbb{P}(\mathcal {E}^{(p^n)})$ with $$\mathcal {O}(T^{\prime})=\mathcal {O}(1)_{\mathbb{P}(\mathcal {E}^{(p^n)})}\otimes\pi_1^{*}\mathcal {L}^{-p^n},$$ which is disjoint from the section $S^{\prime}=F_{2}^{-1}(S)$.
Let $T=F_1^{-1}(T^{\prime})$ denote the (scheme-theoretic) inverse image of $T^{\prime}$ under the $n$-th relative Frobenius morphism $F_1$; local calculations show $T$ is a smooth curve which is called a multiple section. Then $T$ is disjoint with $S$ and  $$\mathcal {O}(T)=\mathcal {O}_{\mathbb{P}(\mathcal {E})}(p^n)\otimes\pi^{*}\mathcal {L}^{-p^n}.$$ 
\subsection{$n$-Raynaud surface}
With the same notations as in the above subsections, we have $\mathcal {O}(S+T)=\mathcal {O}(p^n+1)\otimes\pi^*\mathcal {L}^{-p^n}$. Suppose there exists a positive integer $l$ such that $l\mid p^n+1$ and $l\mid \deg \mathcal {L}$. Let $\mathcal {L}=\mathcal {O}(lN)$ for some divisor $N$ on $C$, and denote   $$d=\frac{p^n+1}{l}.$$
Then $$\mathcal {O}(p^n+1)\otimes\pi^*\mathcal {L}^{-p^n}=\mathcal {M}^l$$
where $$\mathcal {M}=\mathcal {O}(d)\otimes\pi^*\mathcal {O}(-p^nN),$$
and the global section $$S+T\in \Gamma(\mathbb{P}(\mathcal {E}),\mathcal {M}^l)$$ defines an $l$-cyclic cover over $\mathbb{P}(\mathcal {E})$ branched along the divisor $S+T$:
$$\psi:X=\spec \bigoplus_{i=0}^{l-1}\mathcal {M}^{-i}\longrightarrow \mathbb{P}(\mathcal {E}).$$
Then $X$ is called \textbf{$n$-Raynaud surface}, and when $n=1$ it is the classical Raynaud surface.

Let $\tilde{S}$ and $\tilde{T}$ denote the reduced pre-images of the ramification curves $S$ and $T$ respectively, then $$\psi^*(T)=l\tilde{T},\ \psi^*(S)=l\tilde{S},\ \mathcal {O}(\tilde{S}+\tilde{T})=\psi^*(\mathcal {M})\text{ and }\mathcal {O}(\tilde{T})=\mathcal {O}(p^n\tilde{S})\otimes\psi^*\mathcal {O}(-p^nN).\eqno(\clubsuit)$$
Next, we list several properties of the $n$-Raynaud surface $X$. Denote the composition by $$\phi:X\stackrel{\psi}\longrightarrow\mathbb{P}(\mathcal {E})\stackrel{\pi}\longrightarrow C,$$
and by explicit computations we obtain the following:
\begin{prop}\label{}
Let $X$ be an $n$-Raynaud surface over an $n$-Tango curve $C$, with the notations established in this section we have
\begin{itemize}
  \item $(\tilde{S}^{2})=\frac{2 g-2}{p^n l}$;
  \item $\omega_{X}=\mathcal{O}_{X}((p^nl-l-p^n-1) \tilde{S}) \otimes \phi^{*} \mathcal {O}_{C}((p^n+l)N)$;
  \item when $(p,n,l)=(2,1,3)$ or $(p,n,l)=(3,1,2)$, $X$ is a quasi-elliptic surface and in all other cases $\omega_{X}$ is ample;
  \item the morphism $\phi: X \rightarrow C$ is a singular fibration, and every fibre $F$ has a cuspidal singularity at $F \cap \tilde{T}$ locally of the form $x^{l}=y^{p^n}$.
\end{itemize}
\end{prop}


The following lemma is adapted from \cite{Zhe17} and we provide a proof for the reader's convenience.
\begin{lem}\label{lem_zheng}(cf. \cite[Prop3.3]{Zhe17}) With the same notations as above,
for any integer $m=ql+r\geq0$ with $0\leq r\leq l-1$, we have
\begin{enumerate}
  \item $\psi_{*} \mathcal{O}_{X}(-m \tilde{S})=\left(\bigoplus_{i=0}^{r-1} \mathcal {M}^{-i}(-(q+1)S)\right) \oplus\left(\bigoplus_{i=r}^{l-1} \mathcal {M}^{-i}(-qS)\right)$ and
  \item $\psi_{*} \mathcal{O}_{X}(m \tilde{S})=\left(\bigoplus_{i=0}^{l-r-1} \mathcal {M}^{-i}(qS)\right) \oplus\left(\bigoplus_{i=l-r}^{l-1} \mathcal {M}^{-i}((q+1)S)\right)$
\end{enumerate}
\end{lem}
\begin{proof}
First, by the definition of an $l$-cyclic cover we know that $\pi_*\mathcal {O}_X=\mathcal {M}^0 \oplus \mathcal {M}^1 \oplus\cdots\oplus \mathcal {M}^{-l+1}$ forms an
 $\mathcal {O}_{\mathbb{P}(\mathcal {E})}$-algebra with multiplication defined by:
$$\mathcal {M}^{-i_{1}}\otimes \mathcal {M}^{-i_{2}}\rightarrow \mathcal {M}^{-i_{1}-i_{2}} \, \text{and }\, \mathcal {M}^{-l}= \mathcal {M}^0(-S-T)\hookrightarrow \mathcal {M}^0.$$ Here, we denote $\mathcal {O}_{\mathbb{P}(\mathcal {E})}$ by $\mathcal {M}^0$ for notational convenience.
Consider pushing down the following diagram via $\psi$
$$
\xymatrix{
  0 \ar[r]  & \mathcal {O}_{X}(-m(\tilde{S}+\tilde{T})) \ar[d]_{Id} \ar[r]^{} & \mathcal {O}_X(-m\tilde{S}) \ar[d]_{} \ar[r]^{} & \mathcal {O}_X|_{m\tilde{T}} \ar[d]_{} \ar[r]^{} & 0  \\
 0  \ar[r]^{} &\mathcal {O}_{X}(-m(\tilde{S}+\tilde{T}))   \ar[r]^{} & \mathcal {O}_{X}  \ar[d]_{} \ar[r]^{} &\mathcal {O}_X|_{m(\tilde{S}+\tilde{T})}  \ar[d]_{} \ar[r]^{} & 0  \\
  & &\mathcal {O}_X|_{m\tilde{S}} \ar[r]^{\simeq} &\mathcal {O}_X|_{m\tilde{S}}  & .}
$$
By ($\clubsuit$) and  projection formula we obtain
 $$
\xymatrix{
  0 \ar[r]  & \mathcal {M}^{-m}\otimes\psi_*\mathcal {O}_{X} \ar[d]_{} \ar[r]^{} & \psi_*\mathcal {O}_X(-m\tilde{S}) \ar[d]_{} \ar[r]^{} & \psi_*(\mathcal {O}_X|_{m\tilde{T}}) \ar[d]_{} \ar[r]^{} & 0  \\
 0  \ar[r]^{} &\mathcal {M}^{-m}\otimes\psi_*\mathcal {O}_{X}   \ar[r]^-{\hbar} & \psi_*\mathcal {O}_{X}  \ar[d]_{} \ar[r]^{} &\psi_*(\mathcal {O}_X|_{m(\tilde{S}+\tilde{T})} ) \ar[d]_{} \ar[r]^{} & 0  \\
  & &\psi_*(\mathcal {O}_X|_{m\tilde{S}}) \ar[r]^{} &\psi_*(\mathcal {O}_X|_{m\tilde{S}})  &   }
$$
where all the morphisms are $\psi_*\mathcal {O}_X$-mod homomorphisms.

In fact, the morphism $\hbar$ can be explicitly described as follows:
$$
\xymatrix{
  \mathcal {M}^{-ql-r}\ar@{}[r]|{\bigoplus}\ar[drrr]& {\cdots}\ar@{}[r]|{\bigoplus}\ar[drrr]& \mathcal {M}^{-ql-l+1}\ar@{}[r]|{\bigoplus}\ar[drrr]&\mathcal {M}^{-(q+1)l}\ar@{}[r]|{\bigoplus}\ar@{-->}[dlll]& {\cdots}\ar@{}[r]|>{\bigoplus}\ar@{-->}[dlll]&\mathcal {M}^{-(q+1)l-r+1} \ar@{-->}[dlll]\\
  \mathcal {O}_{\mathbb{P}(\mathcal {E})}\ar@{}[r]|{\bigoplus}& {\cdots}\ar@{}[r]|{\bigoplus} & \mathcal {M}^{-r+1}\ar@{}[r]|{\bigoplus}&\mathcal {M}^{-r}\ar@{}[r]|{\bigoplus}& {\cdots}\ar@{}[r]|{\bigoplus}&\mathcal {M}^{-l+1}  . }
$$
Thus, we obtain
$$\psi_* {O}_X|_{m(\tilde{S}+\tilde{T})}=\left(\bigoplus_{i=0}^{r-1} \mathcal {M}^{-i}|_{(q+1)(S+T)}\right) \oplus\left(\bigoplus_{i=r}^{l-1} \mathcal {M}^{-i}|_{q(S+T)}\right).$$
On the other hand, since $S\cap T=\emptyset$, $\psi_* {O}|_{m(\tilde{S}+\tilde{T})}$ decomposes as two direct summands:
$$\psi_* {O}_X|_{m\tilde{S}}=\left(\bigoplus_{i=0}^{r-1} \mathcal {M}^{-i}|_{(q+1)S}\right) \oplus\left(\bigoplus_{i=r}^{l-1} \mathcal {M}^{-i}|_{qS}\right)$$ and
$$\psi_* {O}_X|_{m\tilde{T}}=\left(\bigoplus_{i=0}^{r-1} \mathcal {M}^{-i}|_{(q+1)T}\right) \oplus\left(\bigoplus_{i=r}^{l-1} \mathcal {M}^{-i}|_{qT}\right).$$
Then by the second column of the second diagram of this proof, we obtain
$$\psi_{*} \mathcal{O}_{X}(-m \tilde{S})=\left(\bigoplus_{i=0}^{r-1} \mathcal {M}^{-i}(-(q+1)S)\right) \oplus\left(\bigoplus_{i=r}^{l-1} \mathcal {M}^{-i}(-qS)\right).$$

For the second equality, note that $\psi_{*} \mathcal{O}_{X}(m \tilde{S})=\psi_{*} \mathcal{O}_{X}(((q+1)l-(l-r)) \tilde{S})=\mathcal{O}_{X}((q+1)S)\otimes\psi_{*} \mathcal{O}_{X}(-(l-r) \tilde{S})$, and then it follows from the first equality.

\end{proof}

\section{Base points of adjoint line bundles}

We maintain all notations from the previous section for brevity. Let $C$ be an $n$-Tango curve with an associated divisor $\mathcal {L}=\mathcal {O}(D)=\mathcal {O}(lN)$ and let $X$ be an $n$-Raynaud surface over $C$. In this section
We study linear systems of the form $|m\tilde{S} + \phi^*Q|$, where $m\in \mathbb{N}^+$ and $Q$ is an ample divisor on $C$.
\subsection{Decomposition and module structure}
For $m=lq+r$ with $0\leq r<l$, Lemma \ref{lem_zheng} gives
\begin{align*}
\psi_*\mathcal{O}_X(m\tilde{S} + \phi^*Q)&=\psi_*\mathcal{O}_X(m \tilde{S})\otimes \pi^*(Q) \\
&\cong \left((\bigoplus_{i=0}^{l-r-1} \mathcal {M}^{-i}(qS)) \oplus(\bigoplus_{i=l-r}^{l-1} \mathcal {M}^{-i}((q+1)S))\right) \otimes \pi^*(Q) \\
&\triangleq \mathcal {M}_0 \oplus \mathcal {M}_1 \oplus \cdots \oplus \mathcal {M}_{l-1},
\end{align*}
where $\mathcal {M}_i =  \mathcal {M}^{-i}((q+[\frac{i+r}{l}])S)\otimes \pi^*(Q) $.

Note that it has a natural $\psi_*\mathcal{O}_X$-module structure where $\psi_*\mathcal {O}_X=\mathcal {M}^0 \oplus\cdots\oplus \mathcal {M}^{-l+1}$. And the multiplication is defined by:
$$ \mathcal {M}_{i}\otimes\mathcal {M}^{-1}\rightarrow \mathcal {M}_{i+1}$$ with the canonical inclusion $$\mathcal {M}^{-l}= \mathcal {M}^0(-S-T)\subset \mathcal {M}^0.$$
\subsection{Generation properties of sections}
Consider the natural decomposition
$$H^0(X, \mathcal{O}_X(m\tilde{S} + \phi^*Q))\cong \bigoplus_{i=0}^{l-1}H^0(\mathbb{P}(\mathcal {E}) , \mathcal {M}_i)$$
then we deduce the following key observation.
\begin{lem}(cf. \cite[Corollary 3.3] {GZZ}\label{fujita})\label{basepoint}
With the decomposition above, the following generation properties hold:
\begin{enumerate}
  \item(Non-generation along $T$) The sections in the subspace $$s \in \bigoplus_{i=1}^{l-1}H^0(\mathbb{P}(\mathcal {E}) , \mathcal {M}_i)\subset H^0(X, \mathcal{O}_X(m\tilde{S} + \phi^*Q))$$ cannot generate $\psi_*\mathcal{O}_S(m\tilde{S} + \phi^*Q)$ as a $\psi_*\mathcal{O}_X$-module along the divisor $T$.
  \item(Base point condition) Furthermore, if the line bundle $\mathcal {M}_0$ has a base point $x\in T$ as an $\mathcal{O}_{\mathbb{P}(\mathcal {E})}$-module,
   then all the sections $$s \in \bigoplus_{i=0}^{l-1}H^0(\mathbb{P}(\mathcal {E}) , \mathcal {M}_i)= H^0(X, \mathcal{O}_X(m\tilde{S} + \phi^*Q))$$ fail to generate $\psi_*\mathcal{O}_X(m\tilde{S} + \phi^*Q)$ as $\psi_*\mathcal{O}_X$-module at the point $x\in T$. In other words, $\psi^{-1}(x)$ is a base point of the line bundle $\mathcal{O}_X(m\tilde{S} + \phi^*Q)$ on $X$.
\end{enumerate}
\end{lem}
\begin{proof}
By the preceding argument, we conclude that
$$\psi_*\mathcal{O}_X(m\tilde{S} + \phi^*Q) \cong \mathcal {M}_0 \oplus \mathcal {M}_1 \oplus \cdots \oplus \mathcal {M}_{l-1}$$ is a $\psi_*\mathcal {O}_X=\bigoplus_{i=0}^{l-1} \mathcal {M}^{-i}$-module and the action of the $\mathcal {O}_{\mathbb{P}(\mathcal {E})}$-algebra $\psi_*\mathcal{O}_X$ on the components of $\psi_*\mathcal{O}_S(m\tilde{S} + \phi^*Q)$ into the first term is described as follows:
\begin{itemize}
  \item $\mathcal {M}_{i}\otimes \mathcal {M}^{i-l}=\mathcal {O}(-S-T)\otimes \mathcal {M}_0\subset \mathcal {M}_0$ as a sub-sheaf defined by tensoring with the ideal sheaf $\mathcal {O}(-S-T)$, when $0\leq i\leq l-r-1$;
  \item $\mathcal {M}_{i}\otimes \mathcal {M}^{i-l}=\mathcal {O}(-T)\otimes \mathcal {M}_0\subset \mathcal {M}_0$ as a sub-sheaf defined by tensoring with the ideal sheaf $\mathcal {O}(-T)$, when $l-r\leq i\leq l-1$.
\end{itemize}
So the sections in any component of $\psi_*\mathcal{O}_S(m\tilde{S} + \phi^*Q)$ other than the first cannot generate $\psi_*\mathcal{O}_S(m\tilde{S} + \phi^*Q)$ as $\psi_*\mathcal{O}_X$-module along the divisor $T$.
\end{proof}
\subsection{Base locus on ruled surfaces}
Building on Lemma \ref{basepoint}, we reduce the study of the complete linear system $|m\tilde{S} + \phi^*Q|$ to its first component:
\[
\mathcal{M}_0 = \mathcal{O}_{\mathbb{P}(\mathcal{E})}\left(\left(p^n - 1 - \frac{p^n + 1}{l} + q\right)S\right) \otimes \pi^*\left((p^n + l)N + mQ\right)
\]
To this end, we requires the following technical lemma.

Before proving the lemma, let us first outline the main idea of the proof.
To establish a criterion for the existence of base points of $|\mathcal{M}_0|$ along $T$ on the ruled surface, one might attempt to restrict it to $T$. However the degree of this restriction is typically large due to the special  geometric nature of the multiple section $T$. For instance, when analyzing an adjoint bundle of some ample bundle on $X$, one generally has $\deg\mathcal{M}_0|_T >2g$, rendering effective criteria infeasible. Instead, Our strategy is bypasses this obstacle by constructing a family of sections of the projection $\pi:\mathbb{P}(\mathcal{E})\rightarrow C$. If the restrictions of $\mathcal{M}_0$ to those sections is the same line bundle with a base point $P$ on $C$ and those sections do not intersect along the fiber $\pi^{-1}(P)$, then the base locus of $\mathcal{M}_0$ contains infinite points of the fibre $\pi^{-1}(P)$ and hence the entire fibre $\pi^{-1}(P)$ since the base locus is Zariski-closed. Consequently $|\mathcal{M}_0|$ has base point along $T$ at $\pi^{-1}(P)\cap T$.

\begin{lem}\label{fixpart}
Let $C$ be a smooth projective curve over an algebraically closed field $\mathbf{k}$  equipped with the canonical sheaf $\omega_C$. Let $\mathcal{E}$ be a vector bundle of rank $2$ on $C$.
Suppose that there is a surjective morphism $\sigma_0:\mathcal{E}\twoheadrightarrow \mathcal {L}_0$ where $\mathcal {L}_0$ is a line bundle on $C$ satisfying the following conditions:
\begin{enumerate}
  \item [(1)] $\dim H^0(C,\mathcal{E}^{\vee}\otimes \mathcal {L}_0)\geq2$ and
  \item [(2)] $ H^0(C,\omega_C\otimes {\mathcal {L}_0^{-q}(-Q)})\neq 0$ for some divisor $Q$ of positive degree on $C$ and some integer $q$.
\end{enumerate}
Then there exists a nonempty open subset $C_0\subset C$ such that for any closed point $P\in C_0$, the base locus of the linear system $$\mid\mathcal {O}_{\mathbb{P}( \mathcal{E})}(q)\otimes \pi^*\mathcal {O}(Q+P)\mid$$ contains the fibre $F=\pi^{-1}(P)$, where $\pi:\mathbb{P}(\mathcal{E})\rightarrow C$ is the projection from the ruled surface $\mathbb{P}(\mathcal{E})$ to $C$.
\end{lem}
\begin{proof}
Note that $\mathbb{P}(H^0(C,\mathcal{E}^{\vee}\otimes \mathcal {L}_{0}))$ parametrises all the morphisms $Mor_C(\mathcal {E}, \mathcal {L}_{0})$ up to scalar isomorphisms of $\mathcal {L}_{0}$.
Since surjectivity is an open condition, there is a non-empty open subset $U_0\subset \mathbb{P}(H^0(C,\mathcal {E}^{\vee}\otimes \mathcal {L}_0))$ such that the corresponding morphisms are surjective.
Thus, the sections $\sigma$ of $\pi:\mathbb{P}(\mathcal {E})\rightarrow C$ with $\sigma^*\mathcal {O}(1)\simeq \mathcal {L}_0$ are parametrised by $U_0$.
Given $\dim U_0>0$, we choose another section $(\sigma_0\neq)\sigma_1\in U_0$  and let $\widetilde{U_0}$ denote the intersection of $U_0$ with the line $\mathbb{P}^1\subset  \mathbb{P}(H^0(C,\mathcal {E}^{\vee}\otimes \mathcal {L}_{0}))$ spanned by $\sigma_0$ and $\sigma_1$. Then $\widetilde{U_0}$ is a (an affine) curve.
Since all the sections in $\widetilde{U_0}$ are of the form $k_0\sigma_0+k_1\sigma_1$ with $k_0,k_1\in \mathbf{k}$, the intersection of any two sections in $\widetilde{U_0}$ on the ruled surface is exactly the set $\sigma_0\cap\sigma_1$.

Considering the restriction $\sigma^*(\mathcal {O}(q)\otimes \pi^*(\mathcal {O}(Q)))\simeq \mathcal {L}_0^q(Q)$  to sections in $\widetilde{U_0}$,
since $H^0(C,\omega_C\otimes {\mathcal {L}_0^{-q}(-Q)})\neq0$,  Lemma \ref{basepointoncurve} guarantees a nonempty open subset $C_0\subset C$ such that for any $P\in C_0$, the base locus of $\mid \mathcal {L}_0^q(Q+P)\mid$ contains $P$.
Consequently, there exists a nonempty open subset $C_0\subset C$ such that for any $P\in C_0$ the base locus of $\mid\mathcal {O}(q)\otimes\pi^*\mathcal {O}(Q+P)\mid$ contains the point $\pi^{-1}(P)\cap \sigma (C)$ for every $\sigma \in \widetilde{U_0}$.
Noting that any two sections in $\widetilde{U_0}$  intersect only along $\sigma_0\cap\sigma_1$ on the ruled surface. We may further shrink $C_0$ to ensure $\pi(\sigma_0\cap\sigma_1)\cap C_0=\emptyset$. For a fixed $P$, the base locus of $\mid\mathcal {O}(q)\otimes\pi^*(Q+P)\mid$ then contains infinitely many points $\{\pi^{-1}(P)\cap \sigma (C)\mid \sigma \in \widetilde{U_0}\}$ on the fibre $\pi^{-1}(P)$ and hence contains the entire fibre $\pi^{-1}(P)$ since the base locus is Zariski-closed.
\end{proof}
\begin{lem}\label{basepointoncurve}
Let $C$ be a smooth projective curve over an algebraically closed field  with a line bundle $\mathcal {L}$ on it.
\begin{enumerate}
  \item If $H^0(C, {\mathcal {L}})\neq0$, then there exists a nonempty open subset $U\subset C$ such that $$h^0(C, \mathcal {L} \otimes\mathcal {O}(-x))=h^0(C,\mathcal {L})-1$$ for any closed point $x\in U$.
  \item If $H^0(C,\omega_C\otimes {\mathcal {L}}^{-1})\neq0$, then there exists a nonempty open subset $U\subset C$ such that $$h^0(C, \mathcal {L} \otimes\mathcal {O}(x))=h^0(C,\mathcal {L})$$ for any closed point $x\in U$.
\end{enumerate}
\end{lem}
\begin{proof}
Note that $H^0(C, {\mathcal {L}})\neq0$ and $H^0(C,\omega_C\otimes {\mathcal {L}}^{-1})\neq0$ imply that the base locus  $\texttt{Bs}(|\mathcal {L}|)\varsubsetneqq C$ and $\texttt{Bs}(|\omega_C\otimes {\mathcal {L}}^{-1}|)\varsubsetneqq C$ respectively. Let $U=C\setminus \texttt{Bs}(|\mathcal {L}|)$ and $U=C\setminus \texttt{Bs}(|\omega_C\otimes {\mathcal {L}}^{-1}|)$ in (1) and (2) respectively, then the two statements follow from Riemann-Roch formula immediately.
\end{proof}
\subsection{Base points on n-Raynaud surfaces}
\begin{thm}\label{main}
Let $C$ be an $n$-Tango curve with an associated vector bundle $\mathcal{E}$ of rank $2$, and let $\psi:X\stackrel{l:1}\rightarrow\mathbb{P}(\mathcal {E})$ be the corresponding $n$-Raynaud surface constructed in section \ref{preliminary}. For any integer $m=lq+r \in \mathbb{N}^+$ where $0\leq r<l$,
suppose that there exists a surjective morphism $\sigma_0:\mathcal{E}\twoheadrightarrow \mathcal {L}_0$ to a line bundle $\mathcal {L}_0$  on $C$ satisfying
\begin{enumerate}
  \item [(1)] $\dim H^0(C,\mathcal{E}^{\vee}\otimes \mathcal {L}_0)\geq2$ and
  \item [(2)] $ H^0(C,\omega_C\otimes {\mathcal {L}_0^{-q}(-Q)})\neq 0$ for some divisor $Q$ of positive degree on $C$.
\end{enumerate}
Then $q< p^n$ and there exists a nonempty open subset $C_0\subset C$ such that for every point $P\in C_0$ the ample line bundle $\mathcal{O}_X(m\tilde{S} + \phi^*(Q+P))$ on $X$ has base point $\phi^{-1}(P)\cap \widetilde{T}$.

In particular, if the condition (2) is replaced by
 \begin{enumerate}
   \item [(2*)] $H^0(C,\omega_C\otimes {\mathcal {L}_0^{-(p^n-1-d)}(-(p^n+l)N-Q)})\neq 0$ for some divisor $Q$ of positive degree,
 \end{enumerate}
 then  the same conclusion holds for the adjoint line bundle $\mathcal{O}_X(K_X +r\tilde{S} + \phi^*(Q+P))$.
\end{thm}
\begin{proof}
First we claim that the surjectivity implies that $\deg \mathcal {L}_0\geq \deg \mathcal {L}$. Indeed, pulling back  the non-split sequence
$$0\rightarrow \mathcal{O}_{C}\rightarrow \mathcal {E}\rightarrow \mathcal {L}\rightarrow 0$$
obtained in subsection \ref{ruledsurface} via the $n$-th iterated  Frobenius map yields the splitting sequence
$$\xymatrix@C=0.5cm{
  0 \ar[r] & \mathcal{O}_{C} \ar[r] & F^{n*}\mathcal {E} \ar[r] &\mathcal {L}^{p^n} \ar[r]\ar@/_/[l]_{\tau} & 0 }.$$
Composing $\tau$ with the quotient $F^{n*}(\sigma_0):F^{n*}\mathcal{E}\twoheadrightarrow \mathcal {L}_0^{p^n}$ gives a morphism $\mathcal {L}^{p^n}\rightarrow \mathcal {L}_0^{p^n}$. If $\deg \mathcal {L}_0< \deg \mathcal {L}$, this morphism must be zero, and $\mathcal {L}^{p^n}=\texttt{Ker}(F^{n*}(\sigma_0))$ by the saturation of $\mathcal {L}^{p^n}$, forcing $\mathcal {L}_0\simeq\mathcal {O}_C$. This would imply the sequence splits, contradicting its non-split nature.
Next, note that $\omega_C\simeq \mathcal {L}^{p^n}$, and by condition (2) we deduce $q< p^n$.

For any closed point $P\in C$, consider the push-forward $\mathcal{O}_X(m\tilde{S} + \phi^*(Q+P))$ onto the ruled surface $\mathbb{P}(\mathcal {E})$
$$\psi_*\mathcal{O}_X(m\tilde{S} + \phi^*(Q+P)) \cong \mathcal {M}_0 \oplus \mathcal {M}_1 \oplus \cdots \oplus \mathcal {M}_{l-1},$$
where the first term is $\mathcal {M}_0=\mathcal {O}_{\mathbb{P}( \mathcal{E})}(q)\otimes \pi^*\mathcal {O}(Q+P)$. By Lemma \ref{fixpart} there exists a nonempty open
subset $C_0\subset C$ such that for all closed points $P\in C_0$, the base locus of $\mid\mathcal {O}_{\mathbb{P}( \mathcal{E})}(q)\otimes \pi^*\mathcal {O}(Q+P)\mid$ contains the fibre
$F=\pi^{-1}(P)$, where $\pi:\mathbb{P}(\mathcal{E})\rightarrow C$ is the natural projection. By Lemma \ref{basepoint}, it follows that $\psi^{-1}(F\cap T)=\phi^{-1}(P)\cap \widetilde{T}$ is a base point of the ample line bundle
$\mathcal{O}_X(m\tilde{S} + \phi^*(Q+P))$ on $X$.

For the final statement, recall that the canonical divisor satisfies $K_{X}=(p^nl-l-p^n-1) \tilde{S} + \phi^{*} (p^n+l)N$.
\end{proof}

\begin{rmk}\label{remark}
To apply Theorem \ref{main}, we provide a constructive method to find a suitable line bundle $\mathcal {L}_0$ on an $n$-Tango curve  with an associated vector bundle $\mathcal{E}$.

Let $C$ be an $n$-Tango curve, $\mathcal{E}$ an associated vector bundle on it, and $\psi:X\stackrel{l:1}\rightarrow\mathbb{P}(\mathcal {E})$ an $n$-Raynaud surface, with notations as  in the previous section. First, by $\mathcal{O}_{C}\subset \mathcal {E}\subset F_*^{n}\mathcal {O}_C$ we see that $H^0(C,\mathcal {E})=H^0(C,\mathcal{O}_{C})=\mathbf{k}$.

As described in subsection \ref{n-Tango} , $\mathcal {E}|_{U_i}=\mathcal {O}_{U_i}\cdot 1\oplus\mathcal {O}_{U_i}\cdot \sqrt[p^n]{z_i}$ and the transition matrix of $\mathcal {E}$ is $\left(
                                                                                                           \begin{array}{cc}
                                                                                                             1 & \beta \\
                                                                                                             0 & \alpha \\
                                                                                                           \end{array}
                                                                                                         \right)\in GL(2,\mathcal {O}_{U_1\cap U_2})$ .
Let $(\beta)=\sum\limits_{D_j\notin U_1}a_jD_j+\texttt{others}$, $(\alpha)=\sum\limits_{D_j\notin U_1}b_jD_j+\texttt{others}$ and define $D_0=-\sum\limits_{D_j\notin U_1}\min\{a_j,b_j,0\}D_j$>0.
Suppose the divisor $D_0\in \Gamma(X,\mathcal {O}(D_0))$ is locally defined by the regular functions  $1\in\Gamma(U_1,\mathcal {O}_{U_1}) $ and $\gamma\in\Gamma(U_2,\mathcal {O}_{U_2})$.
Then  $\mathcal {O}(D_0)|_{U_1}=\mathcal {O}_{U_1}\cdot 1$ and $\mathcal {O}(D_0)|_{U_2}=\mathcal {O}_{U_2}\cdot\frac{1}{\gamma}$, and thus $\mathcal {\mathcal {L}}(D_0)|_{U_1}=\mathcal {O}_{U_1}\cdot\eta_1$ and $\mathcal {\mathcal {L}}(D_0)|_{U_2}=\mathcal {O}_{U_2}\cdot\frac{1}{\gamma}\eta_2$, where $\eta_i$ are local bases of $\mathcal{L}$ as in subsection \ref{n-Tango}.
 Therefore, there exists a section $s^{\prime}\in \Gamma(X,\mathcal {\mathcal {L}}(D_0))$,  locally expressed as $s^{\prime}|_{U_1}=1\cdot\eta_1$ and $s^{\prime}|_{U_2}=\gamma \alpha\cdot \frac{1}{\gamma}\eta_2$,
  since $\gamma \alpha \in \Gamma(U_2,\mathcal {O}_{U_2})$ by the construction of $D_0$. Moreover, this section   lifts to a section $s\in \Gamma(X,\mathcal {E}(D_0))$, locally given by $s|_{U_1}=1\cdot \sqrt[p^n]{z_1}$ and $s|_{U_2}=\gamma \alpha\cdot \frac{1}{\gamma}\sqrt[p^n]{z_2}+\gamma\beta\cdot\frac{1}{\gamma}$, since $\gamma \alpha, \gamma \beta \in \Gamma(U_2,\mathcal {O}_{U_2})$ by the construction of $D_0$.
Consequently, $\dim H^0(C,\mathcal {E}^{\vee}\otimes\mathcal{L}(D_0))=\dim H^0(C,\mathcal {E}(D_0))\geq\dim H^0(C,\mathcal{O}_{C}(D_0))+1\geq2$.
Moreover,  by the construction of $D_0$,   $(\gamma \alpha, \gamma\beta)=1$ in the local
ring $\mathcal {O}_{C,x}$ for any closed point $x\notin U_1$. Thus, the inclusion $\mathcal{O}_{C}(-D_0)\hookrightarrow \mathcal {E} $ defined by $s$ is saturated, yielding  a quotient
$\sigma_0:\mathcal{E}\twoheadrightarrow  \mathcal{L}(D_0)$.
\end{rmk}
To conclude this section, we revisit  Example \ref{Tangocurve}  to provide a concise proof of the strong non-freeness of adjoint bundles on $n$-Raynaud surfaces.
\begin{cor}(cf. \cite[Theorem 1.2] {GZZ}\label{fujita})
For any integer $r>0$, there exists a smooth projective surface $X$ with an ample divisor $A$   such that the adjoint linear system
$$|K_X+r A |$$ has base points.
\end{cor}
\begin{proof}
For simplicity,  set $Q(X,Y)=Y^e$ in Example \ref{Tangocurve}. Then
 $\alpha=y^{e(qe-3)}$ and $\beta=Q^{-1}(x,y)=y^{-e}$. Note that $C\setminus U_1=\{\infty\}$. By the construction in Remark \ref{remark}, define $D_0=e\infty$ and $\mathcal{L}_0=\mathcal{L}(D_0)$, then there is a surjective morphism $\sigma_0:\mathcal{E}\twoheadrightarrow \mathcal {L}_0$  satisfying the condition (1) of Theorem \ref{main}. To facilitate calculations, let
 $l=p^n+1$ and $e=kl$ for some integer $k$. Then $\omega_C\otimes {\mathcal {L}_0^{-(p^n-1-d)}(-(p^n+l)N)}=\mathcal{O}(k(q(q+1)k-3-(q-2)(q+1))\infty)$. For any $r>0$, choose $n\gg0$ such that
 $l=p^n+1>r$ and   $k\gg0$ such that $k(q(q+1)k-3-(q-2)(q+1))>r$. Set $Q=(r-1)\infty$; then  $H^0(C,\omega_C\otimes {\mathcal {L}_0^{-(p^n-1-d)}(-(p^n+l)N-Q)})=H^0(C,(k(q(q+1)k-3-(q-2)(q+1))-r+1)\infty)\neq 0$.
 By Theorem \ref{main}, there exists a nonempty open subset $C_0\subset C$ such that for any point
 $P\in C_0$ $\phi^{-1}(P)\cap \widetilde{T}$ is a base point of the adjoint line bundle $\mathcal{O}_X(K_X + r\tilde{S} + \phi^*(Q+P))$. Writing $Q+P=rQ_0$ for some divisor $Q_0$ of degree $1$ on $C$ and letting $A=\tilde{S} + \phi^*(Q_0)$ which is ample, the conclusion follows.
\end{proof}

\section{Strong Kodaira non-vanishing}
Although the Kodaira vanishing theorem fails in positive characteristic, the following weaker analogue of the vanishing theorem remains valid.
\begin{thm}(cf. \cite[Proposition 2.1]{S79})
Let $\mathcal{H}$ be a nef and big line bundle on a smooth projective surface $X$ over $\mathbf{k}$, then $H^1(X,\mathcal{H}^{-n})=0$ for all $n\gg0$.
\end{thm}
However, The following strong non-vanishing theorem asserts that there is no universal bound for $n$.

\begin{thm}\label{strongnonvanishing}
For any integer $m>0$, there exists a smooth projective surface $X$ and an ample line bundle $\mathcal{H}$ on $X$ such that $H^1(X,\mathcal{H}^{-p^m})\neq0$
\end{thm}
\begin{proof}
In the proof, we adopt the notations from section \ref{preliminary}. Let $X$ be an $n$-Raynaud surface over an $n$-Tango curve $C$, with morphisms $$\phi:X\stackrel{\psi}\longrightarrow\mathbb{P}(\mathcal {E})\stackrel{\pi}\longrightarrow C$$
and consider the ample line bundle $\mathcal{H}=\mathcal{O}(\tilde{S} + \phi^*Q)$ where $Q$ is a divisor on $C$ of positive degree.

The Leray spectral sequence $E^{p,q}_{2}=H^p(C,R^q\phi_*\mathcal{H}^{-p^m})\Rightarrow H^{p+q}(X,\mathcal{H}^{-p^m})$, yields the following exact sequence
$$0\rightarrow  H^1(C,\phi_*\mathcal{H}^{-p^m})\rightarrow H^1(X, \mathcal{H}^{-p^m}) \rightarrow H^0(C,R^1\phi_*\mathcal{H}^{-p^m}) \rightarrow H^2(C,\phi_*\mathcal{H}^{-p^m})=0.$$
Writing $p^m=lq+r$ with $0\leq r\leq l-1$, Lemma \ref{lem_zheng} implies
\begin{align*}
\phi_*\mathcal{H}^{-p^m}&=\pi_*\psi_*\mathcal{O}_X(-p^m(\tilde{S}+\phi^*Q)) \\
&\cong \left(\bigoplus_{i=0}^{r-1} \pi_*\mathcal {M}^{-i}((-q-1)S)\otimes \mathcal {O}(-p^mQ)\right) \oplus\left(\bigoplus_{i=r}^{l-1} \pi_*\mathcal {M}^{-i}(-qS)\otimes \mathcal {O}(-p^mQ)\right) \\
&\cong \left(\bigoplus_{i=0}^{r-1} \emph{Sym}^{-id-q-1}(\mathcal{E})\otimes \mathcal {O}(ip^nN-p^mQ)\right) \oplus\left(\bigoplus_{i=r}^{l-1} \emph{Sym}^{-id-q}(\mathcal{E})\otimes \mathcal {O}(ip^nN-p^mQ)\right)=0
\end{align*}
since the exponents of symmetric powers are negative.
Thus, the leftmost term $H^1(C,\phi_*\mathcal{H}^{-p^m})=0$, reducing the problem to computing $H^0(C,R^1\phi_*\mathcal{H}^{-p^m})$.

Note that
$$\omega_{X/C}=\mathcal{O}_{X}((p^nl-l-p^n-1) \tilde{S}) \otimes \phi^{*} \mathcal {O}_{C}((p^n+l-p^nl)N),$$
then by the relative Serre duality, we have
\begin{align*}
(R^1\phi_*\mathcal{H}^{-p^m})^{\vee}
&\simeq \phi_*(\mathcal{H}^{p^m}\otimes \omega_{X/C})\\
&\simeq \pi_*\psi_*\mathcal{O}_X((p^m+p^nl-l-p^n-1)\tilde{S}+\phi^*(p^mQ+(p^n+l-p^nl)N))\\
&\simeq (\bigoplus_{i=0}^{l-r-1} \pi_*\mathcal {M}^{-i}((p^n-d+q-1)S)\otimes \mathcal {O}(p^mQ+(p^n+l-p^nl)N))\oplus\\
&\quad\ (\bigoplus_{i=l-r}^{l-1} \pi_*\mathcal {M}^{-i}((p^n-d+q)S)\otimes \mathcal {O}(p^mQ+(p^n+l-p^nl)N))\\
&\simeq (\bigoplus_{i=0}^{l-r-1} \emph{Sym}^{-id+p^n-d+q-1}(\mathcal{E})\otimes \mathcal {O}(p^mQ+(p^n+l+ip^n-p^nl)N))\oplus\\
&\quad\ (\bigoplus_{i=l-r}^{l-1} \emph{Sym}^{-id+p^n-d+q}(\mathcal{E})\otimes \mathcal {O}(p^mQ+(p^n+l+ip^n-p^nl)N)).
\end{align*}
For the first direct summand, there is a quotient
$$\emph{Sym}^{p^n-d+q-1}(\mathcal{E})\otimes \mathcal {O}(p^mQ+(p^n+l-p^nl)N))\twoheadrightarrow \mathcal {O}(p^mQ+(lq-1)N))$$
by the construction of $\mathcal{E}$.
  For any $m>0$, choose $n=mm_0$ for some odd integer $m_0>0$ and  set $l=p^m+1$. Then $l\mid p^n+1$ and $q=0$.
 Select an $n$-Tango curve $C$ with $\deg N>p^m$, then there exists a divisor $Q$ of degree $1$ such that $N-p^mQ>0$. Thus,
\[
H^0(C, R^1 \phi_* \mathcal{H}^{-p^m}) \supset H^0(C, \mathcal{O}_C(p^m Q + (lq - 1)N))^\vee \neq 0,
\]
implying \( H^1(X, \mathcal{H}^{-p^m}) \neq 0 \).
For instance, the $n$-Tango curves in example \ref{Tangocurve} satisfy those conditions by taking $e=l$ and $Q=\infty$.
\end{proof}

\end{document}